\documentclass{amsart}
\usepackage{amssymb,
amsmath,
amsthm,
graphicx,
amscd,
multind,
eufrak
}
\theoremstyle{plain}
\newtheorem{thm}{Theorem}
\newtheorem{lm}[thm]{Lemma}

\theoremstyle{definition}
\newtheorem{re}[thm]{Remark}
\newtheorem{example}[thm]{Example}

\newcommand{\CC}{{\mathbb C}}
\newcommand{\NN}{{\mathbb N}}

\newcommand{\TT}{{\mathbb T}}
\newcommand{\im}{\operatorname{im}}
\newcommand{\diag}{\operatorname{diag}\nolimits}
\newcommand{\one}{{\mathbf 1}}

\newcommand{\cV}{\mathcal{M}}
\newcommand{\cSV}{\mathcal{SM}}
\newcommand{\cAV}{\mathcal{AM}}
\newcommand{\rk}{\operatorname{rk}}
\newcommand{\la}{\langle}
\newcommand{\ra}{\rangle}
\newcommand{\Wedge}{\bigwedge}
\newcommand{\Sp}{\mathrm{Sp}}
\newcommand{\Crit}{\mathrm{Crit}}
\newcommand{\dto}{\dashrightarrow}

\begin{document}

\title[ML-duality]{Maximum likelihood duality\\ for determinantal varieties}
\author[J.~Draisma]{Jan Draisma}
\address[Jan Draisma]{
Department of Mathematics and Computer Science\\
Technische Universiteit Eindhoven\\
P.O. Box 513, 5600 MB Eindhoven, The Netherlands\\
and Centrum voor Wiskunde en Informatica, Amsterdam,
The Netherlands}
\thanks{The first author is supported by a Vidi grant from
the Netherlands Organisation for Scientific Research (NWO)}
\email{j.draisma@tue.nl}

\author[J.~Rodriguez]{Jose Rodriguez}
\address[Jose Rodriguez]{
Department of Mathematics\\
University of California at Berkeley\\
970 Evans Hall 3840\\
Berkeley, CA 94720-3840 USA
}
\thanks{The second author is supported  by the US National Science Foundation DMS-0943745}
\email{jo.ro@berkeley.edu}

\date{October 2012}

\begin{abstract}
In a recent paper, Hauenstein, Sturmfels, and the second author discovered
a conjectural bijection between critical points of the likelihood function
on the complex variety of matrices of rank $r$ and critical points on the
complex variety of matrices of co-rank $r-1$. In this paper, we prove
that conjecture for rectangular matrices and for symmetric matrices,
as well as a variant for skew-symmetric matrices.
\end{abstract}

\maketitle

\section{Introduction and results}
For an $m \times n$-data table $U=(u_{ij}) \in \NN^{m \times n}$,
we define the {\em likelihood function} $\ell_U: \TT^{m \times n}
\to \TT$, where $\TT=\CC^*$ is the complex one-dimensional torus, as
$\ell_U(Y)=\prod_{ij} y_{ij}^{u_{ij}}$ for $Y=(y_{ij})_{ij} \in \TT^{m
\times n}$. This terminology is motivated by the following observation.
If $Y$ is a matrix with positive real entries adding up to $1$,
interpreted as the joint probability distribution of two random variables
taking values in $[m]:=\{1,\ldots,m\}$ and $[n]:=\{1,\ldots,n\}$,
respectively, then up to a multinomial coefficient depending only on $U$,
$\ell_U(Y)$ is the probability that when independently drawing $\sum_{i,j}
u_{ij}$ pairs from the distribution $Y$, the number of pairs equal to
$(i,j)$ is $u_{ij}$. In other words, $\ell_U(Y)$ is the likelihood of $Y$,
given observations recorded in the table $U$. A standard problem in
statistics is to {\em maximize} $\ell_U(Y)$.

Without further constraints on $Y$ this maximization problem 
is easy: it is uniquely solved by the
matrix $Y$ obtained by scaling $U$ to lie in said probability simplex. But
various meaningful statistical models require $Y$ to lie in
some {\em subvariety} $X$ of $\TT^{m \times n}$. For instance, in
the model where the first and second random variable are required to be
independent, one takes $X$ equal to the intersection of the variety of
matrices of rank $1$ with the hyperplane $\sum_{ij} y_{ij}=1$ supporting
the probability simplex. Taking mixtures of this model, one is also led
to intersect said hyperplane with the variety of rank-$r$ matrices.

For general $X$, the maximum-likelihood estimate is typically
much harder to find (though in the independence model it is still
well-understood). One reason for this is that the restriction of $\ell_U$
to $X$ may have many critical points. Under suitable assumptions,
this number of critical points is finite and independent of $U$ (for
sufficiently general $U$), and is called the {\em maximum likelihood
degree} or {\em ML-degree} of $X$. Finiteness and independence of $U$
hold, for instance, for smooth closed subvarieties of tori.  In that
case, the ML-degree equals the signed topological Euler characteristic
of $X$; see \cite{Huh12} or the more general ``non-compact Riemann-Roch
formula'' in \cite{FK00}. Finiteness and independence of $U$ also hold
for all varieties $X$ studied in this paper \cite{Hauenstein12,Hosten05}
(which are smooth but not closed, and become closed but singular if one
takes the closure).

We take $X$ to be a smooth, irreducible, locally closed,
complex subvariety of a torus. Doing so, we tacitly
shift attention from the statistical motivation to complex
geometry---in particular, we no longer worry whether the critical
points counted by the ML-degree lie in the probability
simplex or are even real-valued matrices. 
Similarly, we also no longer restrict ourselves to study integer
valued tables but allow our data to vary over the complex numbers

The set of all critical points for varying data matrices $U$ has 
a beautiful
geometric interpretation: Given $P \in X$ and a
vector $V$ in the tangent space $T_P X$ to $X$ at $P$, the derivative of
$\ell_U$ at $P$ in the direction $V$ equals $\ell_U(P) \cdot \sum_{ij}
\frac{v_{ij}}{p_{ij}} u_{ij}$. This vanishes if and only if $U$
is perpendicular, in the standard symmetric bilinear form on $\CC^{m
\times n}=\CC^{mn}$, to the entry-wise quotient $\frac{V}{P}$ of $V$
by $P$. This leads us to define
\[ \Crit(X):=\left\{(P,U)\ \middle|\ \frac{T_P X}{P} \perp U \right\}
\subseteq X \times \CC^{m \times n}, \]
which is called {\em the variety of critical points} of $X$ in
\cite{Huh12}, except that there $U$ varies over projective space
and the closure is taken. By construction, $\Crit(X)$ is smooth and
irreducible, and has dimension $mn$; indeed, it is a vector
bundle over $X$ of rank $mn-\dim X$. The ML-degree of $X$ is well-defined
if and only if the projection $\Crit(X) \to \CC^{m \times n}$ is dominant,
in which case the degree of this rational map is the
ML-degree of $X$.

In this paper, motivated by \cite{Hauenstein12}, we consider three choices
for $X$, all given by rank constraints: First, in the {\em rectangular}
case, we order $m,n$ such that $m \leq n$, fix a rank $r \in
[m]$, and take $X$
equal to
\[ \cV_r:=\left\{P \in \TT^{m \times n}\ \middle|\ \sum_{ij} p_{ij}=1
\text{ and } \rk P = r\right\}. \]
Second, in the {\em symmetric} case, we take $m=n$ and take
$X$ equal to 
\[ \cSV_r:=\left\{
P=
\begin{bmatrix}
2p_{11} & p_{12} & \cdots & p_{1m}\\
p_{12} & 2p_{22} &        & \\
\vdots &  & \ddots & \\
p_{1m} &  &  & 2p_{mm}
\end{bmatrix} \in \TT^{m \times m}\
\middle|\ \begin{array}{c}  \sum_{i \leq j} p_{ij}=1\\ \text{
and } \rk(P)=r \end{array}
\right\}.
\]
Third, in the {\em skew-symmetric} or {\em alternating}
case, we take $m=n$ and, for {\em even} $r \in [m]$, take $X$ equal to
\[{\cAV}_{r}:=\left\{ P=\left[\begin{array}{cccc}
0 & p_{12} & \cdots & p_{1m}\\
-p_{12} & 0\\
\vdots &  & \ddots\\
-p_{1m} &  &  & 0
\end{array}\right]\in\mathbb{C}^{m\times m}\ \middle|\ \begin{array}{c}
\sum_{i<j}p_{ij}=1,\\ 
\rk\left(P\right)=r,\\
\text{and }\forall i<j:p_{ij}\neq0
\end{array}\right\} .\]
Minor modifications of the likelihood function are needed in the latter
two cases: we define as $\ell_U(P):=\prod_{i \leq j} p_{ij}^{u_{ij}}$
in the symmetric case, and as $\ell_U(P):=\prod_{i<j} p_{ij}^{u_{ij}}$
in the alternating case.

\begin{table}
\begin{tabular}{c|cccccccc}
        & $(m,n)$ = & (3,3) & (3,4) & (3,5) & (4,4) & (4,5)  & (4,6) & (5,5) \\
\hline
$r=1$ & &      1 &   1   &   1 & 1 &  1 &  1     & 1 \\
$r=2$ & &      10 & 26 & 58 &  191 &  843 &  3119 & 6776 \\
$r=3$ & &        1         &1 &1 & 191 &  843 &   3119 &  61326 \\
$r=4$ & &       & & &1 &1 &1 & 6776 \\
$r=5$ & &  & & & & & & 1 \\
\end{tabular}\\
\ \\
\ \\
\caption{ML-degrees of $\cV_r$ for small values of $r \leq m \leq n$}
\label{eq:MLvalues}
\end{table}

In \cite{Hauenstein12}, using the numerical algebraic geometry software
$\tt{bertini}$ \cite{BHSW06}, the ML-degree of $\cV_r$ is computed
for various values of $r,m,n$ with $r \leq m \leq n$. The numbers are
listed in Table \ref{eq:MLvalues}. Observe that the numbers for rank $r$
and rank $m-r+1$ coincide. The natural conjecture put forward in that
paper is that this always holds \cite[Conjecture 1.2]{Hauenstein12},
and that there is an explicit bijection between the two sets of critical
points \cite[Conjecture 4.2]{Hauenstein12}. Moreover, similar results
were conjectured for symmetric matrices. We will prove these conjectures,
for which we use the term {\em ML-duality} suggested to us by Sturmfels.

\begin{thm}[ML-duality for rectangular matrices and for symmetric matrices]

Fix a rank $r \in [m]$ and let $U \in \NN^{m \times n}$ with $m
\leq n$ ($m=n$ in the symmetric case) be a sufficiently general data
matrix (symmetric in the symmetric case). Then there is an explicit
involutive bijection between the critical points of $\ell_U$ on
$\cV_r$ (respectively, $\cSV_r$) and the critical points of $\ell_U$
on $\cV_{m-r+1}$ (respectively, $\cSV_{m-r+1}$). In particular, the
ML-degrees of $\cV_r$ and $\cV_{m-r+1}$ (respectively, $\cSV_r$
and $\cSV_{m-r+1}$) coincide. Moreover, the product $\ell_U(P)\ell_U(Q)$
is the same for all pairs consisting of a rank-$r$ critical point $P$
and the corresponding rank-$(m-r+1)$ point $Q$.
\end{thm}

In the alternating case, the ML-dual of $\cAV_r$ turns out {\em not}
to be some $\cAV_s$ but rather an affine translate of a determinantal
variety defined as follows. Let $S$ be the skew $m \times
m$-matrix 
\[ S:=\begin{bmatrix}
0 & 1 & \cdots & 1 \\
-1 & 0 & \ddots & \vdots\\
\vdots & \ddots & \ddots & 1\\
-1 & \cdots & -1 & 0 
\end{bmatrix},
\]
and for even $s \in \{0,\ldots,m-1\}$ consider the variety
\[ \cAV'_s:=\left\{P \in \CC^{m \times m}\ \middle|\ P \text{ skew, }
\forall i<j: 
p_{ij} \neq 0, \text{ and } \rk(S-P)=s \right\}.
\]
Note that, unlike in $\cAV_r$, the upper triangular entries
of $P \in \cAV'_s$
are not required to add up to $1$. 

\begin{thm}[ML-duality for skew matrices]
Fix an even rank $r \in \{2,\ldots,m\}$ and let $U \in \NN^{m \times m}$
be a sufficiently general {\em symmetric} data matrix with zeroes on the
diagonal. Let $s \in \{0,\ldots,m-2\}$ be the largest even integer less
than or equal to $m-r$.  Then there is an explicit involutive bijection
between the critical points of $\ell_U$ on $\cAV_r$ and the critical
points of $\ell_U$ on $\cAV'_s$. In particular, the ML-degrees of $\cAV_r$
and $\cAV'_s$ coincide. Moreover, the product $\ell_U(P)\ell_U(Q)$
is the same for all pairs consisting of a rank-$r$ critical point $P$
on $\cAV_r$ and the corresponding rank-$s$ point $Q$ on $\cAV'_s$.
\end{thm}

The proof is similar in each of the three cases. First, we determine the
tangent space to $X$ at a critical point $P$ of $\ell_U$ for sufficiently
general $U$.  It turns out that this space is spanned by certain rank-one
or rank-two matrices. Imposing that $P$ be a critical point, i.e., that
the derivative of $\ell_U$ vanishes in each of these low-rank directions
leads to the conclusion that a certain matrix $Q$, determined from $P$
using some involution involving the fixed matrix $U$, has rank at most
$m-r+1$ (or $s$ in the skew case) and is itself a critical point on the
variety of matrices of its rank. Letting $k \leq m-r+1$ (respectively,
$k \leq s$)  be generic rank of the $Q$s thus obtained, we reverse the
roles of $P$ and $Q$ to argue that $k$ must equal $m-r+1$ (respectively,
$s$), thus establishing the result. In the remainder of this paper we
fill in the details in each of the three cases, in particular making
the involution $P \to Q$ explicit.

\section*{Acknowledgments}

We thank Bernd Sturmfels for his encouragement and comments on early
versions of this paper.

\section{Maximum likelihood duality in the rectangular case}

Let $m \leq n$ be natural numbers and let $\cV_r \subseteq
\TT^{m \times n}$ denote the 
variety of $m \times n$-matrices of rank $r$ whose entries sum
up to $1$. Fix a sufficiently general data matrix
$U=(u_{ij})_{ij} \in \NN^{m \times
n}$, which gives rise to the likelihood function $\ell_U:\cV_r \to \TT,\
\ell_U(P)=\prod_{i,j} p_{ij}^{u_{ij}}$. Let $P \in \cV_r$ be a critical point
for $\ell_U$, which means that the derivative of $\ell_U$ vanishes on
the tangent space $T_P \cV_r$ to $\cV_r$ at $P$. This
tangent space equals 
\begin{equation} \label{eq:Tangent}
T_P \cV_r=\left\{X=(x_{ij})_{ij} \in \CC^{m \times n}\
\middle|\ X \ker P
\subseteq \im P \text{ and } \sum_{ij} x_{ij}=0\right\}. 
\end{equation} 
Here the first condition ensures that $X$ is tangent at $P$ to the
variety of rank-$r$ matrices (see, e.g., \cite[Example 14.6]{Harris92})
and the second condition ensures that $X$ is tangent to the hyperplane
where the sum of all matrix entries is $1$.

Given $X \in T_P \cV_r$, the derivative of $\ell_U$ in that direction
equals $\ell_U(P) \cdot \sum_{ij} \frac{x_{ij}
u_{ij}}{p_{ij}}$, which vanishes if and only if the second
factor vanishes.  We will
now prove that the marginals of $P$ are proportional to those of $U$
(see also \cite[Remark 4.6]{Hauenstein12}).  We write $\one$ for the
all-one vectors in both $\CC^m$ and $\CC^n$, and use self-explanatory
notation such as $u_{i+}:=\sum_j u_{ij}$ and $u_{++}:=\sum_{ij} u_{ij}$.

\begin{lm}\label{lm:SuffStat}
The column vector $P\one$ is a non-zero scalar multiple of $U\one$
and the row vector $\one^T P$ is a non-zero scalar multiple of $\one^T U$.
\end{lm}

\begin{proof}
We prove the first statement; the second statement is proved similarly.
We want to show that the $2 \times 2$-minors of the $m\times 2$-matrix $[P\one |
U\one]$ vanish. We give the argument for the upper minor. Let $X=(x_{ij})$
be the $m \times n$-matrix whose first row equals $p_{2+}$ times the first
row of $P$, whose second row equals $-p_{1+}$ times the second row of $P$,
and all of whose other rows are zero. Then $X \in T_P \cV_r$, so that
the derivative $\sum_{ij} x_{ij} \frac{u_{ij}}{p_{ij}}$ is zero. On the
other hand, substituting $X$ into $\sum_{ij} x_{ij} \frac{u_{ij}}{p_{ij}}$
yields $u_{1+}p_{2+}-u_{2+}p_{1+}$, hence this minor is zero as desired.
The scalar multiple in both cases is $\frac{p_{++}}{u_{++}}=\frac{1}{u_{++}}$,
which is non-zero.
\end{proof}

Define $Q=(q_{ij})_{ij}$ by $p_{ij} q_{ij} = u_{i+} u_{ij} u_{+j}$. This
is going to be our dual critical point, up to a normalization factor
that we determine now.

\begin{lm} \label{lm:Normalization}
The sum $\sum_{ij} q_{ij}$ equals $(u_{++})^3$.
\end{lm}

\begin{proof}
By Lemma~\ref{lm:SuffStat} the rank-one matrix $Y$ defined by
$y_{ij}=u_{i+} u_{+j}$ has image contained in $\im P$. Hence it satisfies
the linear condition $Y \ker P \subseteq \im P$, but not the
condition $\sum_{ij}
y_{ij}=0$. Similarly, $P$ itself satisfies $P \ker P
\subseteq \im P$, but
not $\sum_{ij} p_{ij}=0$. Hence, we can decompose $Y$ uniquely as $cP +
X$ where $c \in \CC$ and where $X$ satisfies $X \ker P \subseteq \im P$
and $\sum_{ij} x_{ij}=0$, i.e., where $X \in T_P \cV_r$. Then we have
\[ \sum_{ij} q_{ij} = \sum_{ij} \frac{y_{ij}
u_{ij}}{p_{ij}} = 
\sum_{ij} c u_{ij} + \sum_{ij} \frac{x_{ij} u_{ij}}{p_{ij}}
=
\sum_{ij} c u_{ij} + 0 = c u_{++} \]
by criticality of $P$. The scalar $c$ equals 
\[ \frac{\sum_{ij} y_{ij}}{\sum_{ij} p_{ij}}=
\frac{\sum_{ij}u_{i+}u_{+j}}{1}=(u_{++})^2, \]
which proves the lemma.
\end{proof}

We will use rank-one matrices in the tangent space $T_P \cV_r$. We equip
both $\CC^m$ and $\CC^n$ with their standard symmetric
bilinear forms.

\begin{lm} \label{lm:Tangent}
The tangent space $T_P \cV_r$ at $P$ is spanned
by all rank-one matrices $v w^T$ satisfying the following two conditions:
\begin{itemize}
\item $v \in \im P$ or $w \bot \ker P$; and \label{it:1}
\item $v \bot \one$ or $w \bot \one$. \label{it:2}
\end{itemize}
\end{lm}

In the proof we will need that $\im P$ is not contained in the hyperplane
$\one^\bot$ and that, dually, $\ker P$ does not contain $\one$. These
conditions will be satisfied by genericity of $U$.

\begin{proof}
The first condition ensures that the rank-one matrices in the lemma
map $\ker P$ into $\im P$, and the second condition ensures that the
sum of all entries of those rank-one matrices is zero, so that they lie
in $T_P \cV_r$, see \eqref{eq:Tangent}.
To show that these rank-one matrices span the tangent space $T_P \cV_r$,
decompose $\CC^m$ as $A \oplus B \oplus C$ where $A \oplus C=\one^\bot$
and $A \oplus B=\im P$. Here we use that $\im P$ is not contained in
the hyperplane $\one^\bot$.

Similarly, decompose $\CC^n=A' \oplus B' \oplus C'$ where $A' \oplus C'$
is the hyperplane $\one^\bot$ and $A' \oplus B'=(\ker P)^\bot$; here
we use the second genericity assumption on $P$.  These spaces have the
following dimensions:
\begin{align*}
\dim A &= r-1 & \dim B &= 1 & \dim C &= m-r\\
\dim A' &= r-1 & \dim B' &= 1 & \dim C' &= n-r.
\end{align*}
The space spanned by the rank-one matrices in the lemma has the space $(B
\otimes B') \oplus (C \otimes C')$ as a vector space complement. The
dimension of this complement is $1+(m-r)(n-r)$, which is also the
codimension of $\cV_r$.
\end{proof}

Let $R=\diag(u_{i+})_i$ and $K=\diag(u_{+j})_j$ be the diagonal matrices
recording the row and column sums of $U$ on their diagonals.  Then,
by Lemma~\ref{lm:SuffStat}, $P\one$ is a scalar multiple of $R\one$
and $\one^T P$ is a scalar multiple of $\one^T K$. This implies that,
in the decompositions in the proof of Lemma~\ref{lm:Tangent}, we may take
$B$ spanned by $U \one = R \one$ and $B'$ spanned by $U \one = K \one$.
Note that $P,Q$ satisfy $P*Q=RUK$, where $*$ denotes the Hadamard product.

Observe also that criticality of $P$ is equivalent to $v^T R^{-1} Q K^{-1}
w=0$ for all rank-one matrices $vw^T$ as in Lemma~\ref{lm:Tangent}. This
criterion will be used in the proof of our duality result for $\cV_r$.

\begin{thm}[ML-duality for rectangular matrices]
\label{thm:Rectangular}
Let $U \in \NN^{m \times n}$ be a sufficiently general data matrix and let
$P$ be a critical point of $\ell_U$ on $\cV_r$. Define $Q=(q_{ij})_{ij}$
by $q_{ij}p_{ij}=u_{i+} u_{ij} u_{+j}$.  Then $Q/(u_{++}^3)$ is a critical
point of $\ell_U$ on $\cV_{m-r+1}$.
\end{thm}

Before proceeding with the proof, we point out that the construction of
$Q':=Q/(u_{++})^3$ from $P$ is symmetric in $P$ and $Q$. As a consequence,
the map $P \mapsto Q'$ from critical points of $\ell_U$ on
$\cV_r$ to critical points on $\cV_{m-r+1}$ is a bijection.  Moreover,
it has the property that $\ell_U(P) \cdot \ell_U(Q')$ depends
only on $U$. In particular, if one lists the critical points $P \in \cV_r$
with positive real entries in order of decreasing log-likelihood, then
the corresponding $Q' \in \cV_{m-r+1}$ appear in order of increasing
log-likelihood, since the sum $\log \ell_U(P)+\log
\ell_U(Q')$ depends only on $U$.

\begin{proof}   
Lemma~\ref{lm:Normalization} takes care of the normalization factor, which
we therefore ignore during most of this proof.  We first show that $Q$
has rank at most $m-r+1$.  For this we take arbitrary $v$ in the space
$A=\bf{1}^{\bot}\cap\im P$ from the proof of Lemma~\ref{lm:Tangent}
and arbitrary $w \in \CC^n$, so that $v w^T \in T_P \cV_r$.  From $v^T
R^{-1} Q K^{-1} w=0$ we conclude that $R^{-1} \im Q \subseteq A^\bot$
because $v$ was arbitrary in $A$.  Equivalently, since $R$ is diagonal
and hence symmetric, we conclude that $\im Q \subseteq(R^{-1} A)^{\bot}$.
The latter space has dimension $m-r+1$, which is therefore an upper
bound on the rank of $Q$.

Similarly, for $w \in A'$ and any $v \in \CC^m$, the matrix $v w^T$ lies
in the tangent space $T_P \cV_r$, and we find $v^T R^{-1} Q K^{-1} w=0$.
Since $v$ was arbitrary, this means that $Q K^{-1} w = 0$, so $\ker Q$
contains $K^{-1} A'$, a space of dimension $r-1$. If $n>m$,
then by the above, $\ker Q$ {\em strictly} contains
$K^{-1} A'$, but this will be irrelevant.

Next we prove that for any rank-one matrix $x y^T$ such that 
\begin{itemize}
\item $x \bot R^{-1} A$ or $y \bot K^{-1} A'$; {\em and}
\item $x \bot \one$ or $y \bot \one$ 
\end{itemize}
we have $\sum_{ij} \frac{x_i u_{ij} y_j}{q_{ij}}=0$.  Note that
the conclusion can be written as $x^T R^{-1} P K^{-1} y=0$, and
observe the similarity with the characterization of $T_P \cV_r$ in
Lemma~\ref{lm:Tangent} that will give us conditions of criticality of $Q$.

Given arbitrary $y \in \CC^n$ we can write $P K^{-1} y$ as $v + cR \one$
with $v \in A$. Then for $x \in (R^{-1} A)^\bot$ perpendicular to $\one$
we find
\[ x^T R^{-1} P K^{-1} y = x^T R^{-1} (v+cR\one)=0 + cx^T \one=0, \]
as desired. If, on the other hand, $x \in (R^{-1} A)^\bot$ is not
perpendicular to $\one$ but $y \in \CC^n$ is, then writing $w:=K^{-1} y$
we have that the vector $v:= Pw$ lies in $A$: indeed, $\one^T v$ is a
scalar multiple of $\one^T Uw$ by Lemma~\ref{lm:SuffStat}, and $\one^T
Uw=\one^T K w = \one^T y=0$.  Hence, again, $x^T R^{-1} P K^{-1} y=x^T
R^{-1} v=0$, since $x \bot R^{-1} A$.  The checks for the case where $y
\bot K^{-1} A'$ are completely analogous.

Now denote the rank of $Q$ by $k$, so that $k \leq m-r+1$. From
$\im Q \subseteq (R^{-1} A)^{\bot}$ and $(\ker Q)^\bot \subseteq
(K^{-1} A')^\bot$ we conclude that the derivative of $\ell_U$ at
$Q'$ in the direction $x y^T$ vanishes, in particular,
when $x y^T$ lies in the tangent space at $Q'$ to
$\cV_k$. Hence $Q'$ is a critical point for $\ell_U$
on $\cV_k$.

Finally, we need to show that the generic rank $k$ of $Q$ thus obtained
(from a sufficiently general $U$ and a critical point $P \in \cV_r$ of
$\ell_U$) equals $m-r+1$, rather than being strictly smaller.  For this,
observe that we have constructed, for any $r \in [m]$, a rational map
of irreducible varieties
\[ \psi_r: \Crit(\cV_r) \dto \Crit(\cV_{f(r)}), \quad
(P,U) \mapsto \left(\frac{1}{(u_{++})^3} \cdot
\frac{RUK}{P}\ , \
U\right)=(Q',U) \]
where $f:[m] \to [m]$ maps $r$ to the generic rank of the matrix $Q'$
as $(P,U)$ varies over $\Crit(\cV_r)$. Since $\psi_r$ commutes with
the projection on the second factor, its image has dimension $mn$,
hence $\psi_r$ is dominant. But it is also injective---in fact, $(P,U)$
can be recovered from $(Q',U)$ with the exact same formula. This shows
that $\psi_r$ is birational, and that $\psi_{f(r)}$ is its inverse
as a birational map. In particular we have $f(f(r))=r$, so that $f$ is a
bijection. But the only bijection $[m] \to [m]$ with the property that
$f(r) \leq m-r+1$ for all $r$ is $r \mapsto m-r+1$. Indeed, if $r$ were
the smallest value for which $f(r) \neq m-r+1$, then $m-r+1$ would not be
in the image of $f$. This concludes the proof of the theorem.  \end{proof}

\begin{re}
It {\em can} happen that the rank of $Q$ is strictly smaller than $m-r+1$
when $U$  is not sufficiently general.  For example, in the rectangular
case where $m=n=4$, if we set
\[
U:=\frac{1}{40}\left[\begin{array}{llll}
4 & 2 & 2 & 2\\
2 & 4 & 2 & 2\\
2 & 2 & 4 & 2\\
2 & 2 & 2 & 4
\end{array}\right]
\text{ and }
\, 
P:=\frac{1}{80}\left[\begin{array}{cccc}
6+2i & 5-\sqrt{5} & 5+\sqrt{5} & 4-2i\\
5-\sqrt{5} & 6-2i & 4+2i & 5+\sqrt{5}\\
5+\sqrt{5} & 4+2i & 6-2i & 5-\sqrt{5}\\
4-2i & 5+\sqrt{5} & 5-\sqrt{5} & 6+2i
\end{array}\right]
\]
then $(P,U)$ lies in $\Crit(\cV_2)$. Also, the number of points in $\Crit(\cV_2)$ with this choice of $U$ is equal to the ML-degree of this model. 
We will find that the number of critical points in $Crit(\cV_3)$  with this non-generic choice of $U$  is $\emph{not}$  equal to the ML-degree of the model. 
Since $u_{++}$=1 we have $Q=Q'$, and 
\[ 
Q=\frac{1}{500}\left[\begin{array}{cccc}
6-2i & 5+\sqrt{5} & 5-\sqrt{5} & 4+2i\\
5+\sqrt{5} & 6+2i & 4-2i & 5-\sqrt{5}\\
5-\sqrt{5} & 4-2i & 6+2i & 5+\sqrt{5}\\
4+2i & 5-\sqrt{5} & 5+\sqrt{5} & 6-2i
\end{array}\right]
\]
satisfies  $p_{ij}q_{ij}=\frac{u_{i+}u_{++}u_{+j}}{u_{++}^{3}}$. In this case, $Q$ 
has rank $2$ instead of rank $3$. This is an important fact for numerical computations.
If we were to to use the homotopy methods as in
\cite{Hauenstein12} to find the critical points of $\ell_{U}$ on ${\cV}_{3}$, we would track a path from a generic point of  $\Crit(\cV_3)$ to the point $(Q,U)$. 
Since $Q$  has rank less than $3$, this will correspond to tracking a path to a singularity leading to numerical difficulties.  
But by determining all critical points of $\ell_U$ on $\cV_{2}$,
we can use duality to avoid these numerical difficulties. That is, to determine the points of
$\Crit(\cV_3)$ with $U$ as above, we use the equation $p_{ij}q_{ij}=\frac{u_{i+}u_{++}u_{+j}}{u_{++}^{3}}$  given by ML-duality and determine which $(q_{ij})$ have rank exactly $3$.
\end{re}

\section{Maximum likelihood duality in the symmetric case}

Let $m$ be a natural number and let $\cSV_r$ denote the variety of
symmetric $m \times m$-matrices of rank $r$ whose entries sum to $2$.
A point $P$ of ${\cSV}_{r}$ and data matrix $U$ will be denoted by
\[
P=\left[\begin{array}{cccc}
2p_{11} & p_{12} & \cdots & p_{1m}\\
p_{12} & 2p_{22}\\
\vdots &  & \ddots\\
p_{1m} &  &  & 2p_{mm}
\end{array}\right]\quad \text{and} \quad U=\left[\begin{array}{cccc}
2 u_{11} & u_{12} & \cdots & u_{1m}\\
u_{12} & 2 u_{22}\\
\vdots &  & \ddots\\
u_{1m} &  &  & 2 u_{mm}
\end{array}\right].
\]
We denote the $(i,j)$-entries of $P$ and $U$ by $P_{ij}$ and $U_{ij}$
to distinguish them from the $p_{ij}$ and $u_{ij}$, respectively.
Recall that the likelihood function in the symmetric case is defined as
$\ell_U(P):=\prod_{i \leq j} p_{ij}^{u_{ij}}$, which in terms of the
entries of $P$ equals $(\prod_{i<j} P_{ij}^{u_{ij}}) \cdot (\prod_i
(P_{ii}/2)^{u_{ii}})$. From now on we fix a sufficiently general data
matrix $U$ and a critical point $P$ for $\ell_U$ on $\cSV_r$.
The tangent space $T_P \cSV_r$ equals 
\begin{equation} \label{eq:TangentSym}
T_P \cSV_r = \left\{X \in \CC^{m \times m} \text{
symmetric}\ \middle|\ X
 \ker P \subseteq \im P \text{ and }\sum_{ij} x_{ij}=0 \right\}. 
\end{equation}
Given
a tangent vector $X\in T_{P}{\cSV}_{r}$, the derivative of $\ell_{U}$
in that direction equals 
\[ \sum_{i<j}\frac{X_{ij}u_{ij}}{P_{ij}} + \sum_i \frac{(X_{ii}/2) u_{ii}}{P_{ii}/2} = \sum_{i \leq j} \frac{X_{ij}u_{ij}}{P_{ij}} \]
(up to a factor irrelevant for its vanishing). 
We set 
\[ 
U_{i+}:=\sum_j U_{ij} \text{ and } U_{++}:=\sum_{i}\sum_{j}
U_{ij},
\]
and similarly for $P$. The symmetric analogue of Lemma~\ref{lm:SuffStat}
is the following.

\begin{lm} \label{lm:SuffStatSym}
The vector $P{\bf 1}$ is a non-zero scalar multiple of $U\one$.
\end{lm}

\begin{proof}
We need to prove that the $m \times 2$-matrix $(P\one|U\one)$ has $2
\times 2$-minors equal to zero. We prove this for the minor in the
first two rows. Set $a:=P_{1+}$ and $b:=P_{2+}$, and define $v_1,v_2
\in \CC^m$ by
$v_1=(b,0,0,\ldots,0)^T, v_2=(0,a,0,\ldots,0)$. 
Let $w_1,w_2$ be the first and second column of $P$, respectively.
Then for each $i=1,2$ the matrix $X^{(i)}=v_i w_i^T + w_i v_i^T$ lies in
the tangent space at $P$ to the variety of symmetric rank-$r$ matrices,
and the difference $X:=X^{(1)}-X^{(2)}$ has sum of entries equal to $0$
and therefore lies in $T_P \cSV_r$. The symmetric matrix $X$ looks like
\[ 
\begin{bmatrix}
2b P_{11} & (b-a)P_{12} & b P_{13} & \cdots & b P_{1m}\\
* & 2a P_{22} & -a P_{23} & \cdots & -a P_{2m}\\
* & * & 0 & \cdots & 0\\
\vdots & \vdots & \vdots & & \vdots\\
* & * & 0 & \cdots & 0 
\end{bmatrix}.
\]
The derivative of $\ell_U$ at $P$ in the direction $X$ equals
\[ \sum_{i \leq j} \frac{X_{ij}u_{ij}}{P_{ij}}
= b U_{1+} - a U_{2+}, \]
and this vanishes by criticality of $P$. The
non-zero scalar is 
$\frac{P_{++}}{U_{++}}=\frac{2}{U_{++}}$.
\end{proof}

The analogue of $R,K$ from the rectangular case is $R:=
\diag\left(U_{1+},\dots,U_{m+}\right),$ which by symmetry of $U$ equals
$\diag\left(U_{+1},\dots,U_{+m}\right)$.  As in the rectangular case,
define the symmetric matrix $Q$ by $P * Q= RUR$, i.e., $P_{ij} Q_{ij} =
U_{i+} U_{ij} U_{j+}$ for $i,j \in [m]$.  This will be our dual critical
point, up to a normalizing factor to be determined now.

\begin{lm} \label{lm:NormalizationSym}
The sum $\sum_{ij} Q_{ij}$ equals $\frac{(U_{++})^3}{2}$.
\end{lm}

\begin{proof}
By Lemma~\ref{lm:SuffStatSym} the rank-one matrix $Y$ with entries
$Y_{ij}=U_{i+}U_{j+}$ has image contained in $\im P$, and so does $P$. So
we can decompose $Y=cP + X$ with $c \in \CC$ and $X \in T_P \cSV_r$, and we find
\[ \sum_{ij} Q_{ij} = \sum_{ij} \frac{Y_{ij} U_{ij}}{P_{ij}} 
= \sum_{ij} c U_{ij} + \sum_{ij} \frac{X_{ij} U_{ij}}{P_{ij}} 
= c U_{++} + 0 = c U_{++}.
\]
Moreover, the scalar $c$ equals
$\frac{Y_{++}}{P_{++}}=\frac{(U_{++})^2}{2}$, which shows
that $Q_{++}=\frac{(U_{++})^3}{2}$.
\end{proof}

As in the rectangular case, we will make use of low-rank elements in
$T_P \cSV_r$, where now ``low rank'' means rank two.

\begin{lm} \label{lm:TangentSym}
The tangent space $T_P \cSV_r$ is spanned by all matrices of the form
$vw^{T}+w^{T}v$ with $v\in im(P)$ and $w\in\mathbb{C}^m$, with the
additional constraint that the sum of all entries is zero, i.e., that
one of $v$ and $w$ is perpendicular to $\bf{1}$.
\end{lm}

In the proof we will implicitly use that $\im P$ is not
contained in $\one^\bot$, which is true by genericity of
$U$.

\begin{proof}
The proof is similar to that of Lemma~\ref{lm:Tangent}. First, the matrices
in the lemma satisfy the conditions characterizing $T_P
\cSV_r$; see
\eqref{eq:TangentSym}. Second, to show that they span that
tangent space, split $\CC^m$ as $A \oplus B
\oplus C$ with $A \oplus B=\im P$
and $A \oplus C=\one^\bot$, so that the second symmetric power $S^2 \CC^m$ equals
\[ S^2(A) \oplus S^2(B) \oplus S^2(C) \oplus (A \otimes B)
\oplus (A \otimes C) \oplus (B \otimes C). \]
The matrices in the lemma span $S^2(A) + A \otimes B + (A \oplus
B)  \otimes C$. This space has dimension $\binom{r}{2}+(r-1)+r(n-r)$,
which equals $\binom{r+1}{2} + r(n-r) -1=\dim \cSV_r$. 
\end{proof}

By Lemma~\ref{lm:TangentSym}, it suffices to understand the derivative
$\sum_{i \leq j} \frac{ X_{ij} u_{ij} }{P_{ij}}$ for $X$ equal to
$vw^{T}+wv^{T}$, in which case it equals
\[
\sum_{i\leq j}\frac{X_{ij}u_{ij}}{P_{ij}}=\sum_{i\leq j}\left(v_{i}w_{j}+w_{i}v_{j}\right)\frac{u_{ij}}{P_{ij}}=v^{T}\left[\begin{array}{cccc}
\frac{2u_{11}}{P_{11}} & \frac{u_{12}}{P_{12}} & \cdots & \frac{u_{1m}}{P_{1m}}\\
\frac{u_{12}}{P_{12}} & \frac{2u_{22}}{P_{22}}\\
\vdots &  & \ddots\\
\frac{u_{1m}}{P_{1m}} &  &  & \frac{2u_{mm}}{P_{mm}}
\end{array}\right]w.
\]
The right-hand side can be concisely written as $v^T(\frac{U}{P})w$, where
$\frac{U}{P}$ is the Hadamard (element-wise) quotient of $U$ by $P$.  So
criticality of $P$ is equivalent to the statement that $v^T(\frac{U}{P})w$
vanishes for all $v,w$ as in Lemma~\ref{lm:TangentSym}.
This, in turn, is equivalent to the condition that
$v^{T}R^{-1}QR^{-1}w=0$
for all $v,w$ as in Lemma~\ref{lm:TangentSym}. We now state
and prove our duality result in the symmetric case. 

\begin{thm}[ML-duality for symmetric matrices] \label{thm:Symmetric}
Let $U \in \NN^{m \times m}$ be a
sufficiently general symmetric data matrix, and let $P$ be a critical point of
$\ell_U$ on $\cSV_r$.  Define the matrix $Q$ by $P_{ij} Q_{ij}=U_{i+}
U_{ij} U_{j+}$. Then $4Q/(U_{++})^3$ is a critical point of $\ell_U$
on $\cSV_{m-r+1}$.
\end{thm}

As in the rectangular case, the map $P \mapsto Q':=4Q/(U_{++})^3$
is a bijection by virtue of the symmetry in $P$ and $Q$, and the same
conclusions for the cricital points with positive real entries can be
drawn as in the rectangular case.

\begin{proof}
The normalizing factor was dealt with in Lemma~\ref{lm:NormalizationSym}
and will be largely ignored in what follows.  As in the proof of
Lemma~\ref{lm:TangentSym}, decompose $\CC^m$ as $A \oplus B \oplus C$
with $A \oplus B=\im P$ and $A \oplus C=\one^\bot$. So $A$ has dimension
$r-1$, $C$ has dimension $m-r$, and $B$ has dimension $1$.  We take $B$
to be spanned by $P\one$, which is a non-zero scalar multiple of $R\one$
by Lemma~\ref{lm:TangentSym}.

First we bound the rank of $Q$. To do so we prove that the image of $Q$
is contained a space of dimension $m-r+1$. Indeed, by criticality of $P$
we have $v^{T}R^{-1}QK^{-1}w=0$ for $w\in\mathbb{C}^{m}$, $v\in\im P$
such that $v\perp{\bf 1}$ or $w\perp{\bf 1}$. Taking $w$ arbitrary and
$v$ in $A$, we find that $\im Q \subseteq (R^{-1} A)^\bot$, which has
dimension $m-r+1$.

Next we show that
\[
x^{T}R^{-1}PK^{-1}y=0
\]
for any $x \in (R^{-1}A)^{\perp}$ and $y \in \CC^m$ with $x\perp{\bf 1}$
or $y\perp{\bf 1}$.
First, suppose $x \bot \one$. Since $PK^{-1}y$ may be written as $a+cR{\bf 1}$ with $a\in A$ and scalar $c$, we find
\[
x^{T}R^{-1}PK^{-1}y=x^{T}R^{-1}a+cx^{T}R^{-1}R{\bf 1}=x^{T}R^{-1}a+0=0.
\]
Otherwise, we have $y \bot \one$ and we may assume $x=cR\one$ with $c$ a 
scalar. In this case, we have
$x^{T}R^{-1}PK^{-1}y=c{\bf 1}^{T}PK^{-1}y$, which by
Lemma~\ref{lm:SuffStatSym} equals a scalar multiple of 
${\bf 1}^TKK^{-1}y={\bf 1}^T y=0$.

Let $k$ be the rank of $Q$. Since $\im Q\subset (R^{-1}A)^{\perp}$
we conclude that $x^{T}R^{-1}PK^{-1}y=0$ holds, in particular,
for all matrices $xy^T+yx^T$ spanning the tangent space to
$\cSV_k$ at $Q'$, so that $Q'$ is critical. By reversing the roles of $P$
and $Q$ and using the involution argument at the end of the proof of
Theorem~\ref{thm:Rectangular}, we conclude that for generic $U$ the value
of $k$ equals $m-r+1$ (rather than being strictly smaller). This proves
the theorem.
\end{proof}

\section{Duality in the skew-symmetric case}

The skew-symmetric case, while perhaps not of direct relevance to
statistics, is of considerable algebro-geometric interest \cite{Hosten05}, since the
variety $\cAV_r$, consisting of skew-symmetric matrices of {\em even}
rank $r$ whose upper-triangular entries are non-zero and add up to $1$,
is (an open subset of a hyperplane section of the affine
cone over) a
secant variety of the Grassmannian of $2$-spaces in $\CC^m$. Recall
that we want to prove that $\cAV_r$ (the {\em intersection} of a determinantal
variety with an affine hyperplane) is ML-dual to the {\em affine translate}
$\cAV'_s$ of a determinantal variety.

A point $P$ of ${\cAV}_{r}$ and data matrix $U$ will be denoted by
\[
P=\left[\begin{array}{cccc}
0 & p_{12} & \cdots & p_{1m}\\
-p_{12} & 0 \\
\vdots &  & \ddots\\
-p_{1m} &  &  & 0  
\end{array}\right]\quad \text{and} \quad
U=\left[\begin{array}{cccc}
0 & u_{12} & \cdots & u_{1m}\\
u_{12} & 0 \\
\vdots &  & \ddots\\
u_{m1} &  &  & 0
\end{array}\right].
\]
Note that $U$ is {\em symmetric} rather than alternating. We
fix a sufficiently general data matrix $U$ and a critical
point $P$ for $\ell_U$ on $\cAV_r$. The tangent space $T_P \cAV_r$ equals 
\[ 
T_P \cAV_r = \left\{X \in \CC^{m \times m} \text{ skew}\ 
\middle|\ 
X\ker P \subseteq \im P \text{ and } \sum_{i<j}
x_{ij}=0\right\}.
\]
The derivative of $\ell_U$ at $P$ in the direction $X$
equals
$\sum_{i<j} \frac{x_{ij} u_{ij}}{p_{ij}}$,
up to a factor irrelevant for its vanishing.
The following lemma is the skew analogue of Lemmas~\ref{lm:SuffStat}
and~\ref{lm:SuffStatSym}.

\begin{lm} \label{lm:SuffStatAlt}
The vector $a=\left(\sum_{j<i} p_{ji} + \sum_{j>i}
p_{ij}\right)_i$ is a scalar
multiple of $U \one$.
\end{lm}

\begin{proof}
We need to show that $2 \times 2$-minors of the matrix $(a | U\one)$
are zero, and do so for the first minor. Let $v_1,v_2$ be the first
and second column of $P$, respectively, and set $w_1:=(a_2,0,\ldots,0)$
and $w_2:=(0,-a_1,0,\ldots,0)$. Then each of the matrices 
$v_i w_i^T-w_i v_i^T$ is tangent at $P$ to the variety of skew-symmetric 
rank-$r$ matrices, and their sum 
\[ 
X=
\begin{bmatrix} 
0 & (a_2-a_1) p_{12} & a_2 p_{13} & \cdots & a_2 p_{1m} \\
-(a_2-a_1) p_{12} & 0 & -a_1 p_{23} & \cdots & -a_1 p_{2m} \\
-a_2 p_{13} & a_1 p_{23} & 0 & \cdots & 0 \\
\vdots & \vdots & \vdots & & \vdots\\
-a_2 p_{1m} & a_1 p_{2m} & 0 & \cdots & 0
\end{bmatrix}
\]
has upper-triangular entries adding up to $0$, so that $X$ is tangent at
$P$ to $\cAV_r$. The derivative of $\ell_U$ at $P$ in the direction $X$, 
which is zero by criticality of $P$, equals 
\[ (a_2-a_1) u_{12} + a_2 u_{13} + \ldots + a_2 u_{1m} 
-a_1 u_{23} - \ldots - a_1 p_{2m} = a_2 u_{1+} - a_1 u_{2+}, \]
which is the minor whose vanishing was required.
\end{proof}

Next we determine rank-two elements spanning $T_P \cAV_r$.
For this we introduce the skew bilinear form $\la .,. \ra$ on $\CC^m$
defined by $\la v,w \ra=v^T S w=\sum_{i<j} (v_i w_j - v_j w_i)$, where
$S$ is the skew-symmetric matrix
\[ S=\begin{bmatrix}
0 & 1 & \cdots & 1\\
-1 & 0 & \ddots &\vdots\\
\vdots &\ddots &\ddots & 1 \\
-1 & \cdots & -1 & 0
\end{bmatrix} \]
from the introduction. 
By elementary linear algebra, this form is non-degenerate if $m$ is
even and has a one-dimensional radical spanned by $(1,-1,1,-1,\ldots,1)
\in \CC^m$ if $m$ is odd.

In what follows, it will be convenient to think of skew-symmetric matrices
also as elements of $\Wedge^2 \CC^m$ or as alternating tensors.

\begin{lm} \label{lm:TangentAlt}
The tangent space $T_P \cAV_r$ is spanned by skew-symmetric matrices of
the form $vw^T - wv^T$ with $v \in \im P$ and $\la v,w \ra=0$.
\end{lm}

In the proof we will use that $\im P$ is non-degenerate with respect to
$\la ., .\ra$. This condition will be satisfied for general $U$.

\begin{proof}
The proof is similar to the symmetric case and the rectangular case:
a skew-symmetric matrix $X$ lies in the tangent space if and only if
$X \ker P \subseteq \im P$ and $\sum_{i<j} x_{ij}=0$. The condition $v
\in \im P$ ensures the first property and the condition that $\la v,
w \ra=0$ ensures the second property.

To complete the proof, decompose $\CC^m$ as $A \oplus C$ with $A=\im P$
and $\la A,C \ra=0$, so that $\Wedge^2 \CC^m$ decomposes as $\Wedge^2
A \oplus (A \otimes C) \oplus \Wedge^2 C$. Taking the vector $w$ in
$v^Tw-wv^T$ from $C$ we see that $A \otimes C$ is contained in the span of
the matrices in the lemma. Next we argue that a codimension-one subspace
of $\Wedge^2 A$ is also contained in their span. Indeed, the (non-zero)
tensors $v^Tw - wv^T \in \Wedge^2 A$ with $v,w \in A$ perpendicular with
respect to $\la .,. \ra$ form a single orbit under the symplectic group
$\Sp(A)=\Sp_r$ (recall that $r$ is even, so that this is a reductive
group), and hence their span is an $\Sp(A)$-submodule of $\Wedge^2 A$.
But $\Wedge^2 A$ splits as a direct sum of only two irreducible modules
under $\Sp(A)$: a one-dimensional trivial module corresponding to (the
restriction of) $\la .,. \ra$ and a codimension-one module. Hence the
tensors $v^T w-w v^T$ must span that codimension-one module.

Summarizing, we find that the matrices in the lemma span a space of 
dimension $r(n-r) + \binom{r}{2} -1$, which equals $\dim \cAV_r$.
\end{proof}

Recall that in the alternating case the likelihood function is given by
$\ell_U(P)=\prod_{i<j} p_{ij}^{u_{ij}}$. The derivative of this expression
in the direction of a skew-symmetric matrix $X$ of the form $vw^T-wv^T$
equals (up to a factor irrelevant for its vanishing)
\[ \sum_{i<j} x_{ij} \frac{u_{ij}}{p_{ij}} 
= \sum_{i<j} \frac{u_{ij}}{p_{ij}} (v_iw_j - v_j w_i) 
= v^T 
\begin{bmatrix}
0 & \frac{u_{12}}{p_{12}} & \cdots & \frac{u_{1m}}{p_{1m}}\\
-\frac{u_{12}}{p_{12}} & 0 & \ddots & \vdots\\
\vdots & \ddots & \ddots & \frac{u_{m-1,m}}{p_{m-1,m}}\\
-\frac{u_{1m}}{p_{1m}} & \cdots & -\frac{u_{m-1,m}}{p_{m-1,m}} & 0 
\end{bmatrix}
w. \]
Define the skew matrix $Q$ by $P*Q=U$. Then criticality of $P$ translates
into $v^T Q w=0$ for all $v \in \im P$ and $w \in \CC^m$ with $\la
v,w \ra=0$.

\begin{thm}[ML-duality for skew matrices] \label{thm:Alternating}
Let $U=(u_{ij})_{ij}$ be a sufficiently general symmetric data
matrix with zeroes on the diagonal, and let $P$ be a critical point
of $\ell_U$ on $\cAV_r$, where $r \in \{2,\ldots,m\}$ is even. Let $s
\in \{0,\ldots,m-2\}$ be the largest even integer less than or equal
to $m-r$. Define the matrix $Q$ by $P * Q=U$.  Then the skew matrix
$Q':=2Q/U_{++}$ is a critical point of $\ell_U$ on the translated
determinantal variety $\cAV'_s$.  Moreover, the map $P \to Q'$ is
a bijection between the critical points of $\ell_U$ on $\cAV_r$ and
those $\cAV'_s$.
\end{thm}

As in the rectangular and symmetric cases, the bijection $P \to Q'$ maps
real, positive critical points to real, positive critical points in such
a way that the sum of the log-likelihoods of $P$ and $Q'$ is constant.

\begin{proof}
By construction of $Q$ we have $v^T Q w=0$ for all $v \in \im P$ and $w
\in \CC^m$ with $v^T S w=0$. This means that the quadratic form $(v,w)
\mapsto v^T Q w$ on $\im P \times \CC^m$ is a scalar multiple of the
quadratic form $(v,w) \mapsto v^T S w$, denoted $\la .,.\ra$
earlier,
on that same space. The scalar is computed by computing
\[ (0,-p_{12},\ldots,-p_{1m}) Q (1,0,\ldots,0)^T = U_{1+} \]
and 
\[ (0,-p_{12},\ldots,-p_{1m}) S (1,0,\ldots,0)^T = P_{1+} =
a_1, \]
where $a$ is the vector of Lemma~\ref{lm:SuffStatAlt}. Using that lemma
and the fact that $\sum_i a_i = 2$ we find that $a_1=2 U_{1+}/U_{++}$. We
conclude that the skew bilinear form associated to $B:=S-\frac{2}{U_{++}}
Q$ is identically zero on $\im P \times \CC^m$, hence $\ker B$ contains
$\im P$ and $\im B=(\ker B)^\bot$ (where $\bot$ refers to the standard
bilinear form on $\CC^m$) is contained in $\ker P=(\im P)^\bot$. In
particular, $B$ has rank at most $s$; let $k \leq s$ denote the actual rank
of $B$.

Next we argue that $Q':=\frac{2}{U_{++}}Q$ is critical for $\ell_U$
on $\cAV'_k$. By arguments similar to (but easier than) those in
Lemma~\ref{lm:TangentAlt} the tangent space $T_{Q'} \cAV'_k$ is spanned
by rank-two matrices $vw^T - wv^T$ with $v \in \im B$ and $w \in \CC^m$
arbitrary. Thus proving that $Q'$ is critical boils down to proving
that $v^T P w=0$ for all $v \in \im B$ and $w \in \CC^m$. But this is
immediate from $\im B \subseteq \ker P$. Thus $Q'$ is critical.

Furthermore, we need to show that (for generic $U$) the rank $k$ of
$B=S-Q'$ is equal to $s$ rather than strictly smaller, and that the map $P
\mapsto Q'$, which is clearly injective, is also surjective on the set of
critical points for $\ell_U$ on $\cAV'_s$. For these purposes we reverse
the arguments above: assume that $Q'$ is a critical point on $\cAV'_k$,
where $k$ is an even integer in the range $\{0,\ldots,m-2\}$. Define
$Q:=\frac{U_{++}}{2} Q'$ and define $P$ by $P*Q=U$. Also, define
$B:=S-Q'$.  Then criticality of $Q'$ implies that $v^T P w=0$ for all
$v \in \im B$ and $w \in \CC^m$, and this implies that $\ker P \supseteq
\im B$. Thus $l:=\rk P$ is at most $m-k$.

Moreover, $B$ itself lies in the tangent space $T_{Q'}
\cAV'_k$, and criticality of $Q'$ implies that $\sum_{i<j} B_{ij}
\frac{U_{ij}}{Q_{ij}}=0$.  Substituting the expression for $B$ into this
we find that
\[ 0 = \sum_{i<j} (1-\frac{2}{U_{++}}Q_{ij})
\frac{U_{ij}}{Q_{ij}} = \sum_{i<j}(P_{ij} -
\frac{2}{U_{++}}) = (\sum_{i<j} P_{ij}) - 1, \]
i.e., the upper-triangular entries of $P$ add up to one. We
conclude that $P$ lies in $\cAV_l$. 
Next, we argue that $P$ is critical. Indeed, for $v \in
\im P$ and $w \in \CC^m$ such that $\la v, w \ra=(v^TSw=)0$ we find 
\[ v^T Q w = v^T(\frac{U_{++}}{2}(S-B))w
=\frac{U_{++}}{2}(v^T S w - v^T B w)=0+0=0, \]
where we have used that $\im P \subseteq \ker B$.

Summarizing, we have found rational maps
\begin{align*} 
\psi_r:\Crit(\cAV_r) &\dto \Crit(\cAV'_{f(r)}), &
(P,U) & \mapsto (\frac{2}{U_{++}} \cdot
\frac{U}{P},U)=(Q',U) \text{ and}\\
\psi'_k:\Crit(\cAV'_k) &\dto \Crit(\cAV_{g(k)}), &
\quad (Q',U) & \mapsto (\frac{2}{U_{++}} \cdot \frac{U}{Q'},U) 
\end{align*}
for some map $f$ mapping even integers $r \in \{2,\ldots,m\}$ to even
integers $k \in \{0,\ldots,m-2\}$, and some map $g$ in the opposite
direction.  By the argument in the proof of Theorem~\ref{thm:Rectangular},
both $\psi_r$ and $\psi'_k$ are birational and $g(f(r))=r$. Hence $f$
is a bijection, and by the above it satisfies $f(r) \leq m-r$. The only such
bijection is the one that maps $r$ to the largest even integer less than
or equal to $m-r$.  This concludes the proof of the theorem.  \end{proof}

\begin{example}
Now we give an explicit example illustrating dual solutions
in the alternating case. For $m=4$  the ML-degree of ${\cAV}_{2}$
is $4$ \cite{Hosten05}. Setting
\[
U=\frac{1}{41}\left[\begin{array}{cccc}
0 & 2 & 3 & 5\\
2 & 0 & 7 & 11\\
3 & 7 & 0 & 13\\
5 & 11 & 13 & 0
\end{array}\right]
\,\text{ and }\, P=\left[\begin{array}{cccc}
0 & 0.0386 & 0.0978 & 0.1075\\
-0.0386 & 0 & 0.1563 & 0.2929\\
-0.0978 & -0.1563 & 0 & 0.3069\\
-0.1075 & -0.2929 & -0.3069 & 0
\end{array}\right],
\]
we have that $P$ is a critical point of $\ell_{U}$ on ${\cAV}_{2}$ and $U_{++}=2$.
Having $Q$ defined as $P*Q=U$, we find that $Q(=Q')$ has full rank. But in
the alternating case the ML-dual variety is an affine
translate of a determinantal variety.  We find that $B=S-Q$ equals
\[
B=\left[\begin{array}{cccc}
0 & -0.2638 & 0.2518 & -0.1344\\
0.2638 & 0 & -0.0924 & 0.0841\\
-0.2518 & 0.0924 & 0 & -0.0332\\
0.1344 & -0.0841 & 0.0332 & 0
\end{array}\right],
\]
and indeed $B$ has rank $4-2=2$. We can actually compute
the ML-degree of $\cAV_2'$ symbolically to be $4$ (even
with the $u_{ij}$ treated as symbols). For the data matrix
$U$ above, the minimal polynomial for $q_{34}$ equals
$434217q_{34}^4-1335767q_{34}^3+1536717q_{34}^2-764049q_{34}+127426$.
\end{example}

\section{Conclusion}

We have proved that a number of natural determinantal varieties of
matrices are {\em ML-dual} to other such varieties living in the
same ambient spaces. However, we have done so without formalizing
what exactly we mean by ML-duality. It would be interesting to find a
satisfactory general definition, perhaps involving the condition that
$(P,U) \mapsto (\frac{U}{P},U)$, or some variant of this
that takes marginals
into account, is a birational map between the two varieties of critical
points. Given such a definition, it would be great to discover new ML-dual
pairs of varieties, for instance so-called {\em subspace varieties}
\cite{Landsberg06} or varieties of consisting of {\em tensors} of given
(border) rank.
Lastly, we note that the problem of finding a formula
for ML-degrees of matrix models remains wide open,
though ML-duality has essentially cut this problem in half.



\end{document}